\documentclass[reqno]{amsart}

\usepackage{mathrsfs}
\usepackage{amsfonts}
\usepackage{amsmath}
\usepackage{amssymb}
\usepackage[table,xcdraw]{xcolor}
\usepackage{pifont}
\usepackage{mathtools}
\usepackage{float}
\usepackage{xcolor}
\usepackage{mathdots}
\usepackage{enumitem}
\usepackage{array}
\usepackage{subcaption}
\usepackage{mathdots}
\usepackage{tikz}
\usepackage{tikz-cd}
\usetikzlibrary{arrows,automata}
\usetikzlibrary{decorations.markings}
\usetikzlibrary{automata, positioning, arrows}
\usetikzlibrary{math}
\usepackage{hyperref}
\usepackage{centernot}

\newcommand{\xmark}{\textcolor{red}{\text{\ding{55}}}}%

\newcommand{\pres}[3]{\textnormal{#1} \langle #2 \mid #3 \rangle}

\newcommand{\C}{\mathbb{C}}
\newcommand{\Z}{\mathbb{Z}}
\newcommand{\Q}{\mathbb{Q}}

\newcommand{\cH}{\mathcal{H}}

\newcommand{\bGamma}{\overline{\Gamma}}
\newcommand{\Mod}[1]{\ (\mathrm{mod}\ #1)}

\newtheorem{innercustomgeneric}{\customgenericname}
\providecommand{\customgenericname}{}

\newtheorem{theorem}{Theorem} 
\numberwithin{theorem}{section}
\newtheorem*{theorem*}{Theorem} 
\newtheorem{lemma}[theorem]{Lemma}     
\newtheorem{corollary}[theorem]{Corollary}

\newtheorem{proposition}[theorem]{Proposition}

\newtheorem*{mainlemma*}{Main Lemma}
\newtheorem{conjecture}{Conjecture}
\newtheorem*{conjecture*}{Conjecture}

\theoremstyle{definition}

\newtheorem{question}{Question}
\newtheorem*{question*}{Question}
\newtheorem{example}{Example}
\numberwithin{example}{section}
\newtheorem{remark}{Remark}

\DeclareMathOperator{\SL}{SL}

\newcommand{\N}{{\mathbb{N}}}
\newcommand{\SLP}[1]{\SL_2(\mathbb{Z}[\frac{1}{#1}])}

\begin{document}

\title[Parabolic congruence subgroups of $\SLP{p}$]{On congruence subgroups of $\SLP{p}$ generated by two parabolic elements}

\author{Carl-Fredrik Nyberg-Brodda}
\address{School of Mathematics, Korea Institute for Advanced Study (KIAS), Seoul 02455, Republic of Korea}
\email{cfnb@kias.re.kr}
\thanks{The author is supported by Mid-Career Researcher Program (RS-2023-00278510) through the National Research Foundation funded by the government of Korea, and by the KIAS Individual Grant (MG094701) at Korea Institute for Advanced Study.}

\subjclass[2020]{30F35, 30F40; Secondary: 20E05, 11J70}

\date{\today}


\keywords{Free groups, Fuchsian groups, Kleinian groups, congruence subgroup.}

\begin{abstract}
We study the freeness problem for subgroups of $\SL_2(\C)$ generated by two parabolic matrices. For $q = r/p \in \Q \cap (0,4)$, where $p$ is prime and $\gcd(r,p)=1$, we initiate the study of the algebraic structure of the group $\Delta_q$ generated by the two matrices
\[
A = \begin{pmatrix}
1 & 0 \\ 1 & 1
\end{pmatrix}, \text{ and } Q_q = \begin{pmatrix}
1 & q \\ 0 & 1
\end{pmatrix}.
\]
We introduce the conjecture that $\Delta_{r/p} = \bGamma_1^{(p)}(r)$, the congruence subgroup of $\SLP{p}$ consisting of all matrices with upper right entry congruent to $0$ mod $r$ and diagonal entries congruent to $1$ mod $r$. We prove this conjecture when $r \leq 4$ and for some cases when $r = 5$. Furthermore, conditional on a strong form of Artin's conjecture on primitive roots, we also prove the conjecture when $r \in \{ p-1, p+1, (p+1)/2 \}$. In all these cases, this gives information about the algebraic structure of $\Delta_{r/p}$: it is isomorphic to the fundamental group of a finite graph of virtually free groups, and has finite index $J_2(r)$ in $\SLP{p}$, where $J_2(r)$ denotes the Jordan totient function. 
\end{abstract}

\maketitle

Consider the following general problem: given a set of matrices, what can be said about the (semi)group $\cH$ generated by them? Can one decide whether some given matrix lies in $\cH$? Can one decide whether $\cH$ is free, or even finitely presented? It should come as no surprise that this type of problem in full generality is undecidable. For example, Markov \cite{Markov1947} proved already in 1947 that the problem of whether a given matrix can be expressed as a product of some other fixed matrices, i.e.\ the membership problem, is undecidable for $\SL_4(\Z)$. This was one of the first undecidability results in mathematics, and demonstrates the inherent difficulty of problems of this type. Nevertheless, in special cases, one can often do much more. 

When studying the matrix group $\SL_n(\mathbf{R})$ for some ring $\mathbf{R}$, there are two degrees of freedom: the size of $n$, and the complexity of $\mathbf{R}$. Even for $n=4$, most algorithmic problems become intractable for $\mathbf{R} = \Z$, as indicated by Markov's result; on the other hand, the group $\SL_3(\Z)$ is still rather mysterious, with many open problems (e.g. coherence, see below). The freeness problem for subgroups of $\SL_3(\Z)$ is highly non-trivial, and has been studied e.g.\ by Zubkov \cite{Zubkov1998}. By contrast, the group $\SL_2(\Z)$ is virtually free and therefore comparatively much simpler to understand. Extending this to a larger ring, we find the group $\SL_2(\Q)$, and, as a particular case, the group $\SLP{p}$ for a fixed prime $p$. Here, similar problems remain as for $\SL_3(\Z)$, and the group $\SLP{p}$ is the starting point for this article. 

The problem central to this article is the following: given two matrices in $\SL_2(\Q)$, do they generate a free group? In fact, we will consider a very particular case of this problem: given a non-zero rational $q \in \Q$, is the (non-abelian) group $\Delta_q$ generated by the two matrices
\[
A = \begin{pmatrix}
1 & 0 \\ 1 & 1
\end{pmatrix}, \text{ and } Q_q = \begin{pmatrix}
1 & q \\ 0 & 1
\end{pmatrix}
\]
free? If not, what is its algebraic structure? These problems are very difficult. The problem of representing free groups by matrices of the type above is a classical topic in combinatorial group theory, cf.\ e.g.\ Magnus \cite[\S5, p. 72]{Magnus1931}, Fuchs-Rabinowitsch \cite{FuchsRabinowitsch1940}, Specht \cite[p. 319]{Specht1959}, and Ree \cite{Ree1961}. The above matrices also appear in generalizations of the \textit{Freiheitssatz} in some one-relator groups with torsion \cite{Ree1968}. Thus, a good deal is known about this problem. For example, when $\lvert q \rvert \geq 4$, the group $\Delta_q$ is known to be free \cite{Brenner1955}, and more generally, if one instead takes $q \in \C$ to be in the \textit{Riley slice} the group $\Delta_q$ is free \cite{Keen1994}, although this remarkable slice is outside the scope of the present article. We also remark that results on the semigroup generated by such matrices have been obtained by several authors \cite{Charnow1974, Brenner1978}.

We will restrict our attention to the case of $q \in \Q$ with $\lvert q \rvert \leq 4$, which by symmetry means we need only consider $q \in \Q \cap (0,4)$. Thus, our problem will be: \textit{is $\Delta_q = \langle A, Q_q \rangle$ ever a free group for $q \in \Q \cap (0,4)$?} This problem was first considered explicitly by Lyndon \& Ullman \cite{Lyndon1969}. Many articles have since treated this problem, and thus far every known rational $q$ in the specified interval gives rise to a non-free group $\Delta_q$. Merzlyakov \cite[Problem~15.83]{Kourovka2023} (cf.\ also Problem~15.84 therein) asks whether the group $\Delta_q$ is ever free for $q \in \Q \cap (0,4)$. Kim \& Koberda \cite{Kim2022} explicitly conjecture that the answer to this question is \textbf{no}, but it remains open in general even to find a dense subset of $\Q \cap (0,4)$ consisting of rational numbers which give rise to non-free $\Delta_q$. It is also not known whether the problem of freeness for $\Delta_q$ is decidable (although it is semi-decidable). Even the problem of whether $\Delta_q$ is finitely presented is an open problem. Lyndon \& Ullman suspected that the group structure of $\Delta_q$ may be as an amalgam of two free groups. If the main conjecture of this article -- Conjecture~\ref{Conj:main-conjecture} -- is true, then the Kim--Koberda conjecture is true, and Lyndon \& Ullman's suspicion is very close to the truth \S\ref{Sec:SL2}).

The theory and machinery of arithmetic matrix groups is sufficiently rich to place the non-freeness conjecture for $\Delta_q$ in a broader context. One particular instance of this is the Greenberg--Shalom Conjecture. This conjecture is rather technical, and we shall not state it; but if it is true, then it would have a number of remarkable consequences, surveyed recently in \cite{Brody2023}. For our purposes, the two most important consequences are that if $r/p \in \Q \cap (0,4)$ with $\gcd(r,p)=1$, then:
\begin{enumerate}
\item $\Delta_{r/p}$ is \textbf{not} free.
\item $\Delta_{r/p}$ has finite index in $\SLP{p}$.
\end{enumerate}
Both problems remain open. The second implies the first, as $\SLP{p}$ is not virtually free. Whereas the first question has been extensively studied \cite{Bamberg2000, Brenner1975, Ignatov1980, Ignatov1980b, Kim2022, Lyndon1969, Smilga2021}, the second remains essentially unexplored (although cf.\ \cite{Detinko2022}). Related to these questions is the following very natural question, which has essentially seen no study beyond the question of freeness:

\begin{question}\label{Quest:delta-rp-structure?}
What is the algebraic structure of $\Delta_{r/p}$?
\end{question}

In this article we will tackle both the finite index question and Question~\ref{Quest:delta-rp-structure?}. We will identify a very compelling conjecture regarding this structure. Specifically, we conjecture, and prove in many cases, that $\Delta_{r/p}$ is equal to $\bGamma_1^{(p)}(r)$, the congruence subgroup of $\SLP{p}$ generated by all matrices with diagonal entries congruent to $1 \pmod{r}$ and upper right entry congruent to $0 \pmod{r}$. Explicitly, we conjecture:

\begin{conjecture}\label{Conj:main-conjecture}
For all $r/p \in \Q \cap (0,4)$ with $\gcd(r,p)=1$, the group $\Delta_{r/p}$ is equal to the congruence subgroup $\bGamma_1^{(p)}(r)$ of $\SLP{p}$. In particular, $\Delta_{r/p}$ is not free, has index $J_2(r)$ in $\SLP{p}$, and is isomorphic to the fundamental group of a finite graph of groups with virtually free vertex groups (and is hence finitely presented).
\end{conjecture}

Here $J_2(r)$ denotes the Jordan totient function counting the number of generators of $(\Z / r\Z)^2$. We will verify the conjecture in some cases; in \S\ref{Sec:final-section}, we prove:

\begin{theorem*}
Let $r/p \in \Q \cap (0,4)$ with $p$ prime, $\gcd(r,p)=1$. If either:
\begin{enumerate}
\item $r \in \{ 1 , 2, 3, 4\}$; or 
\item $r = 5$ and $p \not\equiv -1,1 \pmod{11}$,
\end{enumerate} 
then Conjecture~\ref{Conj:main-conjecture} holds for $\Delta_{r/p}$. Furthermore, if $p$ is arbitrary and 
\[
r \in \{ p-1, p+1, \frac{p+1}{2} \},
\]
the last case assuming $p$ is odd, then assuming a strong form of Artin's conjecture on primitive roots (Conjecture~\ref{Conj:strong-artin}) for $p$ holds, then Conjecture~\ref{Conj:main-conjecture} holds for $\Delta_{r/p}$. 
\end{theorem*}

Our work can be seen as support for two consequences of the Greenberg--Shalom Conjecture. Doubtlessly, the conjecture can be proved in many more cases with analogous techniques as in this article. We have also verified Conjecture~\ref{Conj:main-conjecture} computationally for many small values of $r/p$, including for all values $r/p \leq 4$ when $p \in \{2, 3,5, 7\}$. This is presented in Table~\ref{Tab:cmputational}. We also introduce the notion of a \textit{strong} relation number, which is to $\Delta_{r/p}$ having finite index in $\SLP{p}$ what being a relation number is to $\Delta_{r/p}$ not being free.

Finally, studying the algebraic structure of $2$-generated subgroups of $\SLP{p}$ is also relevant for the \textit{coherence} problem, i.e.\ whether or not every finitely generated subgroup is finitely presented. This is open; however, for prime $p$ the Greenberg--Shalom conjecture implies a positive answer to this question \cite[Theorem~6.2]{Brody2023}. For $n$ with two distinct prime factors $p$ and $q$, an argument by Baumslag \cite[p. 734]{Canberra1974} shows that $\SLP{n}$ is incoherent, in fact the group generated by 
\[
A = \begin{pmatrix}
1 & 0 \\ 1 & 1
\end{pmatrix} \quad \text{and} \quad \begin{pmatrix}
p/q & 0 \\ 0 & q/p
\end{pmatrix}
\]
is not finitely presented. Thus, understanding the structure of $2$-generated subgroups of $\SLP{p}$ seems a worthwhile task; for example, it is not known to the author whether every $2$-generated subgroup of $\SLP{p}$ is finitely presented. Our main conjecture (Conjecture~\ref{Conj:main-conjecture}) implies that the answer to this question is yes at least in the case of the $2$-generated groups $\Delta_{r/p}$.

\section*{Acknowledgement}

I would like to thank Sang-hyun Kim (KIAS) for introducing me to the freeness problem for $\Delta_{r/p}$ and for many helpful discussions at every stage of this project.

\clearpage

\section{$\SL_2$}\label{Sec:SL2}

For a ring $\mathbf{R}$, we denote by $\SL_2(\mathbf{R})$ the ring of all $2 \times 2$-matrices with entries in $\mathbf{R}$ and with determinant $1$. We will only consider the multiplicative structure of $\SL_2(\mathbf{R})$ in this article, and refer to it as a group.

\subsection{Generators and relations}

Throughout this article, we will consider $\SL_2(\Z)$ defined abstractly by the group presentation
\begin{equation}
\SL_2(\Z) = \pres{Gp}{a,b}{b^2 = (ab)^3, b^4=1}.
\end{equation}
We will, however, seamlessly pass between this presentation and considering $\SL_2(\Z)$ as a matrix group, via the isomorphism induced by 
\[a \mapsto A = \begin{pmatrix} 1 & 0 \\ 1 & 1 \end{pmatrix}, \quad b \mapsto B = \begin{pmatrix} 0 & 1 \\ -1 & 0 \end{pmatrix}.\]
Throughout this article, we will retain this notation for the matrices $A$ and $B$. 

Giving generators and relations for $\SLP{p}$ is more difficult, but two such descriptions are known: the first, described by Serre \cite[II.1.4]{Serre1980} comes via the action of $\SLP{p}$ on a tree:
\begin{equation}\label{Eq:amalgam-decomp-SL2Zp}
\SLP{p} \cong \SL_2(\Z) \ast_{\Gamma_0^\Z(p)} \SL_2(\Z)
\end{equation}
where the inclusion map is specified in \cite[Corollary~II.1.2]{Serre1980} (we will not use it). Here $\Gamma_0^\Z(p)$ is used to denote the subgroup of $\SL_2(\Z)$ consisting of all matrices with bottom left entry congruent to $0 \pmod{p}$. Second, Behr \& Mennicke \cite{Behr1968} provide a presentation for all $p$.  For $p=2,3$, the presentation is
\[
\SLP{p} \cong \pres{Gp}{a,b,u_p}{(ab)^3 = (u_p b)^2 = (b u_p a^p)^3 = b^2, b^4 = 1, u_p^{-1} a u = a^{p^2}}
\]
The generators $a$ and $b$ correspond to $A$ and $B$ above, and $u_p$ corresponds to
\[
U_p = \begin{pmatrix}
p & 0 \\ 0 & 1/p
\end{pmatrix}.
\]
Throughout this article, we will retain the notation $U_p$ for this matrix. To see that $A,B,U_p$ generate $\SLP{p}$ is easy, cf.\ \cite[p. 204]{Mennicke1967}. However, we also note that \textit{any} non-trivial power of $U_p$ will generate $\SLP{p}$ together with $A$ and $B$, as an easy multiplication shows that
\begin{equation}\label{Eq:U_k-generates}
U_p = U_p^k A^{p^{k-1}} U_p^{-k} B A^{p^{k+1}} U_p^{-k} B A^{p^{k-1}}B^{-1}
\end{equation}
for any non-zero $k \in \Z$. We will also make consistent use of the following relations, which allow us to ``push'' $U_p$-type letters to the left. 
\begin{equation}\label{Eq:relations-of-SL2Zp-A-B-push}
B^i U_p  = U_p^{(-1)^i} B^i, \qquad AU_p = U_p A^{p^2}.
\end{equation}

\subsection{Congruence subgroups}\label{Subsec:congruence}

In this section, we will collect some known results about congruence subgroups of $\SL_2(\Z)$ and $\SLP{p}$. 

Let $r \geq 1$, and consider the natural surjection $\pi_r^\Z \colon \SL_2(\Z) \to \SL_2(\Z / r \Z)$ resp.\ (when $\gcd(p,r)=1$) the map $\pi_r^{(p)} \colon \SLP{p} \to \SL_2(\Z / r\Z)$ in each case defined by reducing each entry modulo $r$. The kernel of such a map is called a \textit{principal congruence subgroup} (of level $r$) of $\SL_2(\Z)$ resp.\ $\SLP{p}$, and will here be denoted $\ker \pi_r^\Z = \Gamma^\Z(r)$ resp.\ $\ker \pi_r^{(p)} = \Gamma^{(p)}(r)$. Jordan proved in the \textit{Trait\'e} \cite[p. 95, Th\'eor\`eme~123]{Jordan1870} that \begin{equation}\label{Eq:SL2r-order}
\lvert \SL_2(\Z / r\Z) \rvert = r J_2(r) \sim \frac{r^3}{\zeta(3)}
\end{equation}
where $J_2(r)$ denotes the \textit{Jordan totient function}, being the number of pairs $(a,b)$ with $1 \leq a, b \leq r$ with $\gcd(a,b,r)=1$, and $\zeta(3)$ is Ap\'ery's constant. In particular, \eqref{Eq:SL2r-order} also gives the index of the $r$th principal congruence subgroup $\Gamma^\Z(r)$ in $\SL_2(\Z)$ (resp.\ the index of $\Gamma^{(p)}(r)$ in $\SLP{p}$). The sequence of values of $J_2(r)$ begins
\[
1, 3, 8, 12, 24, 24, 48, 48, 72, 72, 120, 96, \dots
\]
being OEIS sequence A007434. Clearly, for $p$ prime $J_2(p) = p^2 - 1$. 

It is well-known that not every subgroup of finite index in $\SL_2(\Z)$ contains some $\Gamma^\Z(r)$ (indeed, most do not), i.e.\ $\SL_2(\Z)$ does not have the \textit{Congruence Subgroup Property}. By contrast, a theorem of Serre \cite{Serre1970} and Mennicke \cite{Mennicke1967}\footnote{Mennicke's proof is incomplete, as it relies on an incorrect computation of the Schur multiplier of $\SL_2(\Z / r\Z)$. The argument is corrected in \cite[\S5]{Beyl1986}.} shows that $\SLP{p}$ has this property for all primes $p$. If $\cH \leq \SLP{p}$ (or $\SL_2(\Z)$) is a congruence subgroup containing $\Gamma^\Z(r)$, then we say that $\cH$ is an $r$-congruence subgroup. We define some important congruence subgroups of $\SL_2(\Z)$ and $\SLP{p}$. Let $r \geq 1$ and $\gcd(r,p)=1$. Then we let 
\[
\bGamma_1^\Z(r) = \left\{ 
\begin{pmatrix}
a & b \\ c & d
\end{pmatrix} \in \SL_2(\Z) \:\:\bigg\vert\:\: a \equiv d \equiv 1 \Mod{r}, b \equiv 0 \Mod{r} \right\}
\]
and analogously for $\bGamma_1^{(p)}(r)$, i.e.\ $\bGamma_1^\Z(r)$ (resp.\ $\bGamma_1^{(p)}(r)$) is the pre-image under $\pi^\Z_r$ (resp.\ $\pi^{(p)}_r$) of the group of lower unitriangular matrices in $\SL_2(\Z/r\Z)$. Since that group is cyclic of order $r$, it follows that 
\begin{equation}
[\SLP{p} : \bGamma_1^{(p)}(r)] = [\SL_2(\Z) : \bGamma_1^{\Z}(r)] = J_2(r).
\end{equation}
The structures of $\bGamma_1^{(p)}(r)$ and $\bGamma_1^\Z(r)$ are quite different, however. It is well-known that for $r \geq 4$ the group $\bGamma_1^\Z(r)$ is free; indeed, denoting its rank by $n$, a simple Euler characteristic argument (see \cite[IX.8]{Brown1982}) yields $\chi(\SL_2(\Z)) = \frac{1-n}{J_2(r)}$ and since $\chi(\SL_2(\Z)) = \zeta(-1) = -\frac{1}{12}$, we have that $n = 1 + \frac{J_2(r)}{12}$. By contrast, $\bGamma_1^{(p)}(r)$ is never free (as $\SLP{p}$ is not virtually free), but by general Bass--Serre theory the decomposition \eqref{Eq:amalgam-decomp-SL2Zp} implies that $\bGamma_1^{(p)}(r)$ is the fundamental group of a finite graph of groups with virtually free vertex groups. We will also prove that $\bGamma_1^{(p)}(r)$ can sometimes (conjecturally always) be generated by three elements in \S\ref{Sec:gamma1-generators}.

\section{The group $\Delta_{r/p}$}\label{Sec:Delta-intro}

For any prime $p$ and $r \in \Z$ such that $p$ and $r$ are coprime. When $\frac{r}{p} \in \Q \cap (0,4)$, we will say that  $r/p$ is \textit{admissible.} Let 
\[
A = \begin{pmatrix}
1 & 0 \\ 1 & 1
\end{pmatrix} \quad \text{and} \quad Q_{r/p} = \begin{pmatrix}
1 & r/p \\ 0 & 1
\end{pmatrix}.
\]
Let $\Delta_{r/p} = \langle A, Q_{r/p} \rangle \leq \SLP{p}$. If $r/p$ is not admissible, then $\Delta_{r/p}$ is free \cite{Sanov1947, Brenner1955}. On the other hand, there is no known example of an admissible (rational) number for which the group $\Delta_{r/p}$ is free; this was noted already by Lyndon \& Ullman \cite{Lyndon1969} in 1969, and Kim \& Koberda \cite{Kim2022} conjectured that \textit{every} admissible rational number gives rise to a non-free $\Delta_{r/p}$. Let us say that $r/p$ is a \textit{relation} number if $\Delta_{r/p}$ is not free (on two generators). We mention a useful necessary and sufficient condition for $r/p$ to be a relation number: that there exist some non-trivial product of the matrices $A$ and $Q_{r/p}$, involving both $A$ and $Q_{r/p}$ non-trivially, which equals an upper triangular matrix. Some $Q_{r/p}$-conjugate of this matrix will then commute with $Q_{r/p}$, proving non-freeness (cf.\ \cite[Lemma~2.1]{Kim2022}). 

There are many known examples of rational relation numbers. We refer the reader to Kim \& Koberda \cite{Kim2022} and Gilman \cite{Gilman2008} for references. Kim \& Koberda, in particular, prove that $r/p$ is a relation number for all $r \leq 27$, except possibly $r = 24$. The main difficulty in proving theorems of this type is that the shortest length of a relation (in the sense of the syllable length of the product of $A$ and $Q_{r/p}$ that is required for a product to equal $1$) grows very rapidly as $r/p \to 4$. There is also no known subset of $\Q \cap (0,4)$ in which the relation numbers are dense. Furthermore, there is no known example of a sequence of distinct relation numbers converging to $3$ or to $4$. By contrast, it is easy to prove that the algebraic relation numbers (defined analogously) are dense in $(0,4)$, see e.g.\ Chang, Jennings \& Ree \cite[Corollary~1]{Chang1958} or \cite[Corollary~1.5(ii)]{Slanina2016}. It is also easy to construct a series of rational relation numbers converging to $2 + \sqrt{2} = 3.414...$. For example, consider the sequence
\begin{equation}\label{Eq:convs-to-2}
\frac{3}{1}, \frac{7}{2}, \frac{41}{12}, \frac{239}{70}, \frac{1393}{408}, \frac{8119}{2378} \dots \longrightarrow 2 + \sqrt{2}
\end{equation}
consisting entirely of principal convergents to $2 + \sqrt{2}$, i.e.\ $2$ added to the solutions to the Pell equation $X^2 + 2Y^2 = 1$. Then it is not hard to prove that the $n$th entry $q_n$ of the sequence will be such that $Q_{q_n}^{a(n)}(A^{-1}Q_{q_n})^3A^{-1}Q_{q_n}^{a(n)}$ is upper triangular, where $a(1) = 0, a(2)= 2$ and recursively $a(n) = 34 a(n-1) - a(n-2) - 8$ for $n \geq 3$. Thus the numbers in \eqref{Eq:convs-to-2} are all rational relation numbers.

In spite of the extensive literature on the freeness problem, going back 50+ years, there has been essentially no research done regarding the algebraic structure of the group $\Delta_{r/p}$. This article aims to fill that gap. In \S\ref{Sec:final-section}, we will describe the structure explicitly in many cases, and prove that in these cases the group $\Delta_{r/p}$ is a congruence subgroup of $\SLP{p}$, and in particular that $\Delta_{r/p}$ has finite index in $\SLP{p}$. Indeed, we will show in many cases that $\Delta_{r/p} = \bGamma^{(p)}_1(r)$, as defined in \S\ref{Sec:SL2}. Among many other consequences of this general conjecture, including the resolution of the Kim--Koberda Conjecture, it would also imply e.g.\ that $\Delta_{r/p}$ is the fundamental group of a finite graph of groups with virtually free vertex groups. This sort of decomposition for $\Delta_{r/p}$ was predicted already by Lyndon \& Ullman in 1969 (cf.\ \cite[p.\ 1401]{Lyndon1969}). We remark that the inclusion $\Delta_{r/p} \leq \bGamma_1^{(p)}(r)$ is trivially true; both $A$ and $Q_{r/p}$ clearly lie in $\bGamma_1^{(p)}(r)$. This leads us directly to Conjecture~\ref{Conj:main-conjecture}. We will deal with one easy case of the conjecture, first, namely when $r=1$. 

\begin{proposition}\label{Prop:r=1}
The group generated by $A$ and $Q_{1/p}$ has index $1$ in $\SLP{p}$. That is, $\Delta_{1/p} = \bGamma_1^{(p)}(1) = \SLP{p} \cong \SL_2(\Z) \ast_{\Gamma_0^\Z(p)} \SL_2(\Z)$. 
\end{proposition}
\begin{proof}
This is a simple consequence of the two equalities $B = A^{-1} Q_{1/p}^p A^{-1}$ and $U_p = B^{-1} Q_{1/p} A^{-p}Q_{1/p}$. Hence, $\SLP{p} = \langle A, B, U_p \rangle = \langle A, Q_{1/p} \rangle$. 
\end{proof}

A particular consequence of Conjecture~\ref{Conj:main-conjecture} is that $\Delta_{r/p}$ has index $J_2(r)$ in $\SLP{p}$. Since $J_2(1) = 1$, we have now verified at least this trivial instance of the index formula. The conjecture can, of course, also be verified computationally using the presentation for $\SLP{p}$ given in \cite{Behr1968}. The results of this for small values of $r/p$ is presented in Table~\ref{Tab:cmputational}. This is obviously very compelling computational evidence in favour of the conjecture. Before continuing to more general $r \geq 2$, we remark that there is no dependency on $p$ in the statement of Proposition~\ref{Prop:r=1}; it appears to be the numerator which plays a dominating role in the structure of $\Delta_{r/p}$.

\begin{table}
\begin{tabular}{l|lll|l|l|l|l|}
\cline{2-6} \cline{8-8}
                             & \multicolumn{1}{l|}{$p=2$}                    & \multicolumn{1}{l|}{$p=3$}                    & $p=5$                    & $p=7$  & $p=11$ &  & $J_2(r)$ \\ \cline{1-6} \cline{8-8} 
\multicolumn{1}{|l|}{$r=1$}  & \multicolumn{1}{l|}{$1$}                      & \multicolumn{1}{l|}{$1$}                      & $1$                      & $1$    & $1$    &  & $1$      \\ \cline{1-6} \cline{8-8} 
\multicolumn{1}{|l|}{$r=2$}  & \multicolumn{1}{l|}{\xmark}                   & \multicolumn{1}{l|}{$3$}                      & $3$                      & $3$    & $3$    &  & $3$      \\ \cline{1-6} \cline{8-8} 
\multicolumn{1}{|l|}{$r=3$}  & \multicolumn{1}{l|}{$8$}                      & \multicolumn{1}{l|}{\xmark}                   & $8$                      & $8$    & $8$    &  & $8$      \\ \cline{1-6} \cline{8-8} 
\multicolumn{1}{|l|}{$r=4$}  & \multicolumn{1}{l|}{\xmark}                   & \multicolumn{1}{l|}{$12$}                     & $12$                     & $12$   & $12$   &  & $12$     \\ \cline{1-6} \cline{8-8} 
\multicolumn{1}{|l|}{$r=5$}  & \multicolumn{1}{l|}{$24$}                     & \multicolumn{1}{l|}{$24$}                     & \xmark                   & $24$   & $24$   &  & $24$     \\ \cline{1-6} \cline{8-8} 
\multicolumn{1}{|l|}{$r=6$}  & \multicolumn{1}{l|}{\xmark}                   & \multicolumn{1}{l|}{\xmark}                   & $24$                     & $24$   & $24$   &  & $24$     \\ \cline{1-6} \cline{8-8} 
\multicolumn{1}{|l|}{$r=7$}  & \multicolumn{1}{l|}{$48$}                     & \multicolumn{1}{l|}{$48$}                     & $48$                     & \xmark & $48$   &  & $48$     \\ \cline{1-6} \cline{8-8} 
\multicolumn{1}{|l|}{$r=8$}  & \multicolumn{1}{l|}{\cellcolor[HTML]{C0C0C0}} & \multicolumn{1}{l|}{$48$}                     & $48$                     & $48$   & $48$   &  & $48$     \\ \cline{1-1} \cline{3-6} \cline{8-8} 
\multicolumn{1}{|l|}{$r=9$}  & \multicolumn{1}{l|}{\cellcolor[HTML]{C0C0C0}} & \multicolumn{1}{l|}{\xmark}                   & $72$                     & $72$   & $72$   &  & $72$     \\ \cline{1-1} \cline{3-6} \cline{8-8} 
\multicolumn{1}{|l|}{$r=10$} & \multicolumn{1}{l|}{\cellcolor[HTML]{C0C0C0}} & \multicolumn{1}{l|}{$72$}                     & \xmark                   & $72$   & $72$   &  & $72$     \\ \cline{1-1} \cline{3-6} \cline{8-8} 
\multicolumn{1}{|l|}{$r=11$} & \multicolumn{1}{l|}{\cellcolor[HTML]{C0C0C0}} & \multicolumn{1}{l|}{$120$}                    & $120$                    & $120$  & \xmark &  & $120$     \\ \cline{1-1} \cline{3-6} \cline{8-8} 
\multicolumn{1}{|l|}{$r=12$} & \cellcolor[HTML]{C0C0C0}                      & \multicolumn{1}{l|}{\cellcolor[HTML]{C0C0C0}} & $96$                     & $96$   & $96$   &  & $96$     \\ \cline{1-1} \cline{4-6} \cline{8-8} 
\multicolumn{1}{|l|}{$r=13$} & \cellcolor[HTML]{C0C0C0}                      & \multicolumn{1}{l|}{\cellcolor[HTML]{C0C0C0}} & $168$                    & $168$  & $168$  &  & $168$    \\ \cline{1-1} \cline{4-6} \cline{8-8} 
\multicolumn{1}{|l|}{$r=14$} & \cellcolor[HTML]{C0C0C0}                      & \multicolumn{1}{l|}{\cellcolor[HTML]{C0C0C0}} & $144$                    & \xmark & $144$  &  & $144$    \\ \cline{1-1} \cline{4-6} \cline{8-8} 
\multicolumn{1}{|l|}{$r=15$} & \cellcolor[HTML]{C0C0C0}                      & \multicolumn{1}{l|}{\cellcolor[HTML]{C0C0C0}} & \xmark                   & $192$  & $192$  &  & $192$    \\ \cline{1-1} \cline{4-6} \cline{8-8} 
\multicolumn{1}{|l|}{$r=16$} & \cellcolor[HTML]{C0C0C0}                      & \multicolumn{1}{l|}{\cellcolor[HTML]{C0C0C0}} & $192$                    & $192$  & $192$  &  & $192$    \\ \cline{1-1} \cline{4-6} \cline{8-8} 
\multicolumn{1}{|l|}{$r=17$} & \cellcolor[HTML]{C0C0C0}                      & \multicolumn{1}{l|}{\cellcolor[HTML]{C0C0C0}} & $288$                    & $288$  & $288$  &  & $288$    \\ \cline{1-1} \cline{4-6} \cline{8-8} 
\multicolumn{1}{|l|}{$r=18$} & \cellcolor[HTML]{C0C0C0}                      & \multicolumn{1}{l|}{\cellcolor[HTML]{C0C0C0}} & $216$                    & $216$  & $216$  &  & $216$    \\ \cline{1-1} \cline{4-6} \cline{8-8} 
\multicolumn{1}{|l|}{$r=19$} & \cellcolor[HTML]{C0C0C0}                      & \multicolumn{1}{l|}{\cellcolor[HTML]{C0C0C0}} & $360$                    & $360$  & $360$  &  & $360$    \\ \cline{1-1} \cline{4-6} \cline{8-8} 
\multicolumn{1}{|l|}{$r=20$} & \cellcolor[HTML]{C0C0C0}                      & \cellcolor[HTML]{C0C0C0}                      & \cellcolor[HTML]{C0C0C0} & $288$  & $288$  &  & $288$    \\ \cline{1-1} \cline{5-6} \cline{8-8} 
\multicolumn{1}{|l|}{$r=21$} & \cellcolor[HTML]{C0C0C0}                      & \cellcolor[HTML]{C0C0C0}                      & \cellcolor[HTML]{C0C0C0} & \xmark & $440$  &  & $440$    \\ \cline{1-1} \cline{5-6} \cline{8-8} 
\multicolumn{1}{|l|}{$r=22$} & \cellcolor[HTML]{C0C0C0}                      & \cellcolor[HTML]{C0C0C0}                      & \cellcolor[HTML]{C0C0C0} & $360$  & \xmark &  & $360$    \\ \cline{1-1} \cline{5-6} \cline{8-8} 
\multicolumn{1}{|l|}{$r=23$} & \cellcolor[HTML]{C0C0C0}                      & \cellcolor[HTML]{C0C0C0}                      & \cellcolor[HTML]{C0C0C0} & $528$  & $528$  &  & $528$    \\ \cline{1-1} \cline{5-6} \cline{8-8} 
\multicolumn{1}{|l|}{$r=24$} & \cellcolor[HTML]{C0C0C0}                      & \cellcolor[HTML]{C0C0C0}                      & \cellcolor[HTML]{C0C0C0} & $384$  & $384$  &  & $384$    \\ \cline{1-1} \cline{5-6} \cline{8-8} 
\multicolumn{1}{|l|}{$r=25$} & \cellcolor[HTML]{C0C0C0}                      & \cellcolor[HTML]{C0C0C0}                      & \cellcolor[HTML]{C0C0C0} & $600$  & $600$  &  & $600$    \\ \cline{1-1} \cline{5-6} \cline{8-8} 
\multicolumn{1}{|l|}{$r=26$} & \cellcolor[HTML]{C0C0C0}                      & \cellcolor[HTML]{C0C0C0}                      & \cellcolor[HTML]{C0C0C0} & $504$  & $504$  &  & $504$    \\ \cline{1-1} \cline{5-6} \cline{8-8} 
\multicolumn{1}{|l|}{$r=27$} & \multicolumn{3}{l|}{\cellcolor[HTML]{C0C0C0}}                                                                            & $684$  & $684$  &  & $684$    \\ \cline{1-6} \cline{8-8} 
\end{tabular}

\vspace{0.2cm}

\caption{Some computational evidence in favour of the main conjecture (Conjecture~\ref{Conj:main-conjecture}). Each box contains the index of $\Delta_{r/p}$ in $\SLP{p}$, computed using GAP. In each instance, we also verified that $\Delta_{r/p} = \bGamma_1^{(p)}(r)$. The blank (grey) boxes correspond to values $r/p \geq 4$. The red crosses $\xmark$ correspond to integers $r/p$, and in those cases $\Delta_{r/p}$ has infinite index in $\SLP{p}$. }
\label{Tab:cmputational}
\end{table}

\clearpage

\section{A sufficient criterion for congruence subgroups}\label{Sec:technical-cosets}

Throughout this section, fix a prime $p$. Let $\cH \leq \SLP{p}$. We say that $\cH \leq \SLP{p}$ is a \textit{weak congruence subgroup} if $\cH \cap \SL_2(\Z)$ is a congruence subgroup of $\SL_2(\Z)$. In particular, $\SL_2(\Z)$ is a weak congruence subgroup of $\SLP{p}$. A (right) coset $\cH W$ of $\cH$ is an \textit{integral} coset if there is some $M \in \SL_2(\Z)$ such that $M \in \SL_2(\Z)$. We will say that $\cH$ has an \textit{integral right coset transversal} (in $\SLP{p}$), or simply that $\cH$ is a \textit{coset-integral} subgroup (of $\SLP{p}$), if $\cH$ has a complete right coset transversal consisting entirely of integer cosets. That is, $\cH$ is a coset-integral subgroup if and only if $\SLP{p} = \cH \cdot \SL_2(\Z)$ as sets. 

As noted by Mennicke \cite[p.\ 220]{Mennicke1967}, $U_p$ is a redundant generator modulo $\ker \pi^{(p)}_r$ for any $r \in \N$ with $\gcd(r,p)=1$. Specifically, we have 
\begin{equation}\label{Eq:U_p-in-quotient}
U_p = A^{p^{-1}} B A^p B A^{p^{-1}} B^{-1} \mod{\ker \pi^{(p)}_r},
\end{equation}
Hence, using \eqref{Eq:U_p-in-quotient} together with the normality of $\ker \pi^{(p)}_r$, we observe that every congruence subgroup of $\SLP{p}$ is coset-integral. The converse, when restricted to weak congruence subgroups, is also true:

\begin{lemma}\label{Lem:weak-congruence-congruence-iff-coset-integral}
Let $\cH$ be a weak congruence subgroup of $\SLP{p}$. Then $\cH$ has only finitely many integral cosets. In particular, $\cH$ is a congruence subgroup if and only if $\cH$ is an coset-integral subgroup.
\end{lemma}
\begin{proof}
Every congruence subgroup is coset-integral by using \eqref{Eq:U_p-in-quotient}. For the converse, suppose $\cH$ is a coset-integral weak congruence subgroup. Any two disjoint integral cosets of $\cH$ in $\SLP{p}$ are also two disjoint integral cosets for $\cH \cap \SL_2(\Z)$ in $\SL_2(\Z)$. Since $\cH$ is a weak congruence subgroup, there are only finitely many of the latter, and so there are finitely many of the former; i.e.\ $\cH$ is a finite index subgroup of $\SLP{p}$. By the congruence subgroup property for $\SLP{p}$, $\cH$ is a congruence subgroup in $\SLP{p}$. 
\end{proof}

\begin{lemma}\label{Lem:pushing-lemma}
Let $\cH$ be a weak congruence subgroup in $\SLP{p}$. Then for all $M \in \SL_2(\Z)$ and all $t \in \Z$, there exists $N \in \SL_2(\Z)$ and $\varepsilon \in \{ -1, 1 \}$ such that
\[
\cH MU^t = \cH U^{\varepsilon t} N.
\]
\end{lemma}
\begin{proof}
For ease of notation, we will only consider the case $t = 1$, the other cases being entirely analogous. Let $M = A^{i_1} B^{j_1} \cdots A^{i_k} B^{j_k} \in \SL_2(\Z)$. We prove the claim by induction on $k$. If $k=0$, there is nothing to prove; we can take $\varepsilon = 1, N = \mathbf{1}_{2\times 2}$. Suppose that $k \geq 1$. Then 
\begin{align}\label{Eq:pushing-u-in-cosets}
\cH M &= \cH A^{i_1} B^{j_1} \cdots A^{i_k} B^{j_k} U_p  = \cH A^{i_1} B^{j_1} \cdots A^{i_k} U_p^{(-1)^{j_k}} B^{j_k}
\end{align}
by \eqref{Eq:relations-of-SL2Zp-A-B-push}. If $(-1)^{j_k} = 1$, then as $A^{i_k} U_p = U_p A^{p^2 i_k}$, the claim follows by induction. So suppose $(-1)^{j_k} = -1$. In the case that $k=1$, then \eqref{Eq:pushing-u-in-cosets} becomes $\cH A^{i_1} U_p^{-1}$. Since $\cH$ is a weak congruence subgroup, say $\cH \cap \SL_2(\Z)$ is an $r$-congruence subgroup of $\SL_2(\Z)$ with $\gcd(p,r)=1$, it follows that $A^{rm} \in \cH$ for all $m \in \Z$. Since $\gcd(r,p^2) = 1$, $r$ is an (additive) generator of $\Z / p^2 \Z$, so there is some $m \in \Z$ such that $rm + i_1 \equiv 0 \pmod{p^2}$. Choose such an $m$, and let $rm + i_1 = sp^2$, where $s \in \Z$. Then 
\[
\cH A^{i_1} U_p^{-1} = \cH A^{rm} A^{i_1} U_p^{-1} = \cH A^{sp^2} U_p^{-1} = \cH U_p^{-1} A^{s},
\]
again using \eqref{Eq:relations-of-SL2Zp-A-B-push}. Thus we are done in this case. Assume, then, that $j_{k-1}$ is non-zero, for otherwise we would be done by induction. 

For all $m \in \Z$, right multiplication by $B^{j_{k-1}}A^{rm}B^{-j_{k-1}}$, this being an element of $\Gamma^\Z(r)$, which is contained in the normal core of $\cH \cap \SL_2(\Z)$, will fix any right coset of $\cH \cap \SL_2(\Z)$ in $\SL_2(\Z)$. Therefore, right multiplication by $B^{j_{k-1}}A^{rm}B^{-j_{k-1}}$ also fixes any integral right coset of $\cH$ in $\SLP{p}$. Consequently, picking $m$ such that $rm + i_k = sp^2$ for some $s \in \Z$, we have
\begin{align*}
\cH A^{i_1} B^{j_1} \cdots A^{i_k} U_p^{-1} B^{j_k} &=  \cH A^{i_1} B^{j_1} \cdots A^{j_{k-1}} (B^{j_{k-1}}A^{rm}B^{-j_{k-1}})B^{j_{k-1}} A^{i_k} U_p^{-1} B^{j_k} \\
&= \cH A^{i_1} B^{j_1} \cdots A^{j_{k-1}} B^{j_{k-1}} A^{sp^2} U_p^{-1} B^{j_k} \\
&= \cH A^{i_1} B^{j_1} \cdots A^{j_{k-1}} B^{j_{k-1}} U_p^{-1}A^{s} B^{j_k}
\end{align*}
Notice that in pushing our $U_p^{\pm 1}$ to the left, we have only changed its sign. By induction, we are done. 
\end{proof}

\begin{lemma}\label{Lem:weak-congruence-with-power-is-integral}
Let $\cH$ be a weak congruence subgroup in $\SLP{p}$. If $U_p^k \in \cH$ for some non-zero $k \in \Z$, then every right coset of $\cH$ is integral.
\end{lemma}
\begin{proof}
Let $W = W(A,B,U_p)$ be a word in $A,B,U_p$ and their inverses. By \eqref{Eq:U_k-generates}, there is no loss of generality to assume that $W = W(A,B,U_p^k)$ is a word in $A,B,U_p^k$. If $W$ contains no occurrence of $U_p^{\pm k}$, then there is nothing to show, as $\cH W$ is integral by definition. Suppose instead that $W$ contains at least one occurrence of $U_p^{\pm k}$ and write $W \equiv M U_p^{\pm k} W'$, where $M \in \SL_2(\Z)$ and $W' = W'(A,B,U_p^k)$ contains one fewer occurrences of $U_p^{\pm k}$. Then by Lemma~\ref{Lem:pushing-lemma} there exists some $N \in \SL_2(\Z)$ and some $\varepsilon \in \{ -1, 1\}$ such that
\[
\cH W = \cH M U_p^{\pm k} W' = \cH U_p^{\varepsilon k} N W' =  \cH N W'.
\]
Now $NW'$ contains one less occurrence of $U_p^k$ than $W$, so by induction on the number of occurrences of $U_p^{\pm k}$ in $W$, the claim follows.
\end{proof}

Thus, by combining Lemma~\ref{Lem:weak-congruence-congruence-iff-coset-integral} and Lemma~\ref{Lem:weak-congruence-with-power-is-integral}, we find:

\begin{proposition}\label{Prop:WC-is-congruence-iff-it-contains-Uk}
Let $\cH \leq \SLP{p}$. Then $\cH$ is a congruence subgroup if and only if it is a weak congruence subgroup and $U_p^k \in \cH$ for some non-zero $k \in \N$.
\end{proposition}
\begin{proof}
$(\implies)$ This follows directly from the fact that congruence subgroups have finite index in $\SLP{p}$. Indeed, if $\cH$ is an $r$-congruence subgroup, then by the pigeonhole principle we have $U_p^{rJ_2(r)} \in \cH$. 

$(\impliedby)$ Suppose $\cH$ and $k \in \N$ are as given. Then by Lemma~\ref{Lem:weak-congruence-with-power-is-integral} we have that $\cH$ is coset-integral, so by Lemma~\ref{Lem:weak-congruence-congruence-iff-coset-integral} $\cH$ is a congruence subgroup.
\end{proof}

\begin{corollary}\label{Cor:index-formula}
Let $\cH \leq \SLP{p}$. If $\cH \cap \SL_2(\Z)$ is an $r$-congruence subgroup of $\SL_2(\Z)$ and $U_p^{k} \in \cH$, for some non-zero $k \in \Z$, then
\[
[ \SLP{p} : \cH] = [\SL_2(\Z) : \cH \cap \SL_2(\Z)] \leq  r J_2(r).
\]
\end{corollary}
\begin{proof}
The first equality follows directly: $(\geq)$ is always true, since $\cH \cap \SL_2(\Z) \leq \cH$, and $(\leq)$ follows from the fact that any congruence subgroup is coset-integral. Finally, the index of $\Gamma^\Z(r)$ in $\SL_2(\Z)$ is well-known to be given by $r J_2(r) = \lvert \SL_2(\Z / r \Z) \rvert$.
\end{proof}

In the case that $\cH \cap \SL_2(\Z) = \SL_2(\Z)$, then we can choose $r=1$, and since $r J_2(r) = 1$ we find that
\[
\SLP{p} = \langle \SL_2(\Z), U_p^{k} \rangle = \langle A, B, U_p^k \rangle.
\]
for all non-zero $k$, as observed already in \eqref{Eq:U_k-generates} (note, however, that \eqref{Eq:U_k-generates} was used in the proof of Corollary~\ref{Cor:index-formula}). In other words, any subgroup of $\SLP{p}$ which contains $\SL_2(\Z)$ is either equal to $\SL_2(\Z)$ or else is all of $\SLP{p}$. 

Recall the definition of $\sigma_p(r)$ as the multiplicative order of $p$ mod $r$. Then in particular we have:

\begin{theorem}\label{Thm:If-contains-GZ-and-power-then-contains-Gp}
Let $\cH \leq \SLP{p}$ be any subgroup such that (1) $\bGamma_1^\Z(r) \leq \cH$ and (2) $U_p^{k\sigma_r(p)} \in \cH$ for some non-zero $k \in \Z$. Then $\bGamma_1^{(p)}(r) \leq \cH$. 
\end{theorem}

We will now use this result to verify that $\Delta_{r/p} = \bGamma_1^{(p)}(r)$ for $r \leq 4$, and use some number-theoretic arguments to give a conjectural generating set for $\bGamma_1^{(p)}(r)$.

\section{Generating $\bGamma_1^{(p)}(r)$}\label{Sec:gamma1-generators}

In this section, we will give a generating set for the congruence subgroup $\bGamma_1^{(p)}(r)$ of $\SLP{p}$, when $\gcd(r,p) = 1$ and $r \leq 4$. Then, conditional on a slightly strengthened form of Artin's conjecture on primitive roots, we will give a generating set for all $r$ with $\gcd(r,p)=1$. First, the following is a restatement of the results of \S\ref{Sec:technical-cosets}, specifically of Theorem~\ref{Thm:If-contains-GZ-and-power-then-contains-Gp}:

\begin{theorem}\label{Thm:r-finite-index-small-r}
Let $p$ be any prime, and $1 \leq r \leq 4$ with $\gcd(r,p)=1$. Then the group
\[
\cH = \left\langle A = \begin{pmatrix}
1 & 0 \\ 1 & 1
\end{pmatrix}, Q_{r/p} = \begin{pmatrix}
1 & r/p \\ 0 & 1
\end{pmatrix}, U_p^{\sigma_p(r)} = \begin{pmatrix}
p^{\sigma_p(r)} & 0 \\ 0 & p^{-\sigma_p(r)}
\end{pmatrix} \right\rangle
\]
is equal to $\bGamma_1^{(p)}(r)$, where $\sigma_p(r)$ denotes the multiplicative order of $p$ mod $r$ (being equal to either $1$ or $2$). 
\end{theorem}
\begin{proof}
It is well-known that $\bGamma_1^\Z(r)$ is generated by $A$ and $Q_{r/p}^p = BA^rB^{-1}$ if (and only if!) $r \leq 4$. The result then follows from Corollary~\ref{Thm:If-contains-GZ-and-power-then-contains-Gp}, as we clearly have $\cH \leq \bGamma_1^{(p)}(r)$.
\end{proof}

We will now give a conjectural generating set for $\bGamma_1^{(p)}(r)$ for arbitrary $r$ with $\gcd(r,p)=1$. Let us recall Artin's conjecture on primitive roots: 

\begin{conjecture*}[Artin]
Let $m$ be any integer which is not a square nor $-1$. Then there are infinitely many primes $q$ such that $m$ is a primitive root mod $q$. 
\end{conjecture*}

The conjecture remains open even for $m=2$; for a survey, see \cite{Moree2012}. Let us consider the following strengthening of Artin's conjecture and Dirichlet's theorem:

\begin{conjecture}[Residual Artin's Conjecture]\label{Conj:strong-artin}
Let $m$ be any integer which is not a square nor $-1$. Let $a, b \in \Z$ be such that $\gcd(a,b)=1$. If the arithmetic progression $a\Z + b$ contains infinitely many primes $q$ such that $m$ is not a quadratic residue mod $q$, then $a\Z + b$ contains infinitely many primes $q$ such that $m$ is primitive mod $q$. 
\end{conjecture}

Note that one obvious (``local'') obstruction to $m$ being primitive mod $q$ is if $m$ is a quadratic residue mod $q$. For example, $2$ is a quadratic residue mod $q$ if and only if $q \equiv \pm 1$ mod $8$. Hence, $2$ is never primitive mod $q$ if $q \in 8\Z\ \pm 1$. However, Conjecture~\ref{Conj:strong-artin} says that there are infinitely many primes $q \equiv 3 \pmod{8}$ such that $2$ is primitive mod $q$, and likewise for $q \equiv 5 \pmod{8}$. More generally, whether or not $m$ is a quadratic residue mod $q$ is always determined by a congruence condition on $q$ mod $km$ for some $k \in \Z$. Hence, by the quantitative version of Dirichlet's theorem, if we impose that $a$ is coprime with $m$ then Conjecture~\ref{Conj:strong-artin} implies that the sequence $a\Z + b$ must contain infinitely many primes $q$ such that $m$ is primitive mod $q$. That is, Conjecture~\ref{Conj:strong-artin} implies the following: let $m$ be any integer which is not a square nor $-1$, and let $a, b \in \Z$ be non-zero with $\gcd(a,b) = 1$ and $\gcd(a,m)= 1$. Then the arithmetic progression $a\Z + b$ contains infinitely many primes $q$ such that $m$ is primitive mod $q$. We will now use it in this form. 

\begin{theorem}\label{Thm:r-finite-index-number-theory-Artin}
Let $p$ be any prime. Assume the strengthened form of Artin's conjecture on primitive roots (Conjecture~\ref{Conj:strong-artin}) holds for $p$. Then for every $r \in \Z$ with $\gcd(r,p)=1$ and for every non-zero $k \in \Z$ we have that the group
\[
\cH_k = \left\langle A = \begin{pmatrix}
1 & 0 \\ 1 & 1
\end{pmatrix}, Q_{r/p} = \begin{pmatrix}
1 & r/p \\ 0 & 1
\end{pmatrix}, U_p^{k} = \begin{pmatrix}
p^{k} & 0 \\ 0 & p^{-k}
\end{pmatrix} \right\rangle
\]
contains $\bGamma_1^{(p)}(r)$. In particular, $[\SLP{p} : \cH_k] \leq J_2(r)$.
\end{theorem}
\begin{proof}
Throughout, we will let $Q_r = Q_{r/p}^p$. Let $M_0 = \begin{psmallmatrix}a_0 & b_0r \\ \ast & \ast \end{psmallmatrix} \in \bGamma_1^{(p)}(r)$. We begin with some housekeeping. First, it is no loss of generality to assume that $a_0, b_0 \in \Z$, for we can always left multiply and conjugate $M_0$ by a sufficiently high power of $U_p^k$ to ensure this. Then we have $\gcd(a_0,b_0) = p^n$ for some $n \in \N$. Hence, since $\gcd(r,p)=1$, right multiplying by a sufficiently large power of $Q_r$ can be used to ensure that $\gcd(b_0,p) =1$ and hence also that $b_0$ is odd. Fix a $b = b_0$ with this property.

Suppose first that $r$ is not divisible by $4$. Then, as right multiplying $M_0$ by $A^n$ transforms the upper left entry to $a_n = a_0 + nbr$, and since $b$ is odd and $\gcd(r,4) \leq 2$, we can by Dirichlet's theorem on arithmetic progressions ensure $a_n$ is an odd prime congruent to $3 \pmod{4}$, and since we assume Conjecture~\ref{Conj:strong-artin}, we may also assume $p$ is primitive mod $a_n$. Then $p^2$ generates the multiplicative subgroup of quadratic residues mod $a$. Furthermore, the multiplicative subgroup of quadratic residues mod $a$ has odd order as $a \equiv 3 \pmod{4}$, and hence, by making $n$ sufficiently large, we may assume $p^{2k}$ also generates this subgroup. Fix an $n$ such that all the above holds, and let $a = a_n$. Now since $a \equiv 3 \pmod{4}$, we have $\left(\frac{-1}{a}\right) = -1$. This fixes our $a$ and $b$; we write $M = \begin{psmallmatrix}a & br \\ \ast & \ast \end{psmallmatrix}$, and of course $M_0 \in \cH_k$ if and only if $M \in \cH_k$. We now prove this last statement. Conjugating our matrix $M$ by $U_p^k$ gives
\[
U_p^k  \begin{pmatrix}
a & br \\ \ast & \ast
\end{pmatrix} U_p^{-k} = \begin{pmatrix}
a & p^{2k} br \\ \ast & \ast
\end{pmatrix}.
\]
Since $p^{2k}$ generates the subgroup of quadratic residues mod $a$, which has index $2$ in $(\Z / a\Z)^\times$, it follows that $p^{2k \ell} \equiv \left(\frac{b}{a}\right) \pmod{a}$ for some $\ell \in \Z$, and hence conjugating by a sufficiently high power will transform $M$ into
\begin{equation}\label{Eq:conjugate-to-apmr}
\begin{pmatrix}
a & \left(\frac{b}{a}\right) r \\ \ast & \ast
\end{pmatrix} = \begin{pmatrix}
a & \pm r \\ \ast & \ast
\end{pmatrix},
\end{equation}
since $\gcd(a,b) = 1$. Since $M_0, M \in \bGamma_1^{(p)}(r)$, we have $a \equiv 1 \pmod{r}$, and hence we can now transform the matrix in \eqref{Eq:conjugate-to-apmr} using the Euclidean algorithm and right multiplication by $A, Q_r$, to obtain 
\[
\begin{pmatrix}
1 & 0 \\ \ast & \ast
\end{pmatrix} = \begin{pmatrix}
1 & 0 \\ s/p^n & 1
\end{pmatrix}, \quad s, n \in \Z.
\]
Conjugating this matrix by a sufficiently high power of $U_p^k$, if necessary, transforms it into a power of $A$, and we are done. 

Finally, suppose that $r$ is divisible by $4$. Then the first part of our approach fails, for we will always have $a \equiv 1 \pmod{4}$. On the other hand, we will therefore always have $-a \equiv 3 \pmod{4}$, and thus proceeding in the same manner as before, we can transform  our matrix $M_0$ into 
\[
\begin{pmatrix}
-a & \left(\frac{b}{a}\right) r \\ \ast & \ast
\end{pmatrix} = \begin{pmatrix}
-a & \pm r \\ \ast & \ast
\end{pmatrix} \to \begin{pmatrix}
1 & 0 \\ \ast & \ast
\end{pmatrix} 
\]
where the last step is again effected by conjugation by $U_p^k$. Arguing as before, we again find that $M_0 \in \cH_k$, which completes the proof. 
\end{proof}

As a corollary, we obtain a three-generator generating set for the congruence subgroup $\bGamma_1^{(p)}(r)$ by noticing that $U_p^{k\sigma_p(r)}$ always lies in $\bGamma_1^{(p)}(r)$ for all non-zero $k \in \Z$, where $\sigma_p(r)$ denotes the multiplicative order of $p$ mod $r$:

\begin{corollary}\label{Cor:r-main-generators-artin}
Let $p$ be a prime. Assume the strengthened form of Artin's conjecture on primitive roots (Conjecture~\ref{Conj:strong-artin}) holds for $p$. Then for any $r$ with $\gcd(p,r) = 1$ and any non-zero $k \in \Z$ we have 
\[
\bGamma_1^{(p)}(r) = \left\langle \begin{pmatrix}
1 & 0 \\ 1 & 1
\end{pmatrix}, \begin{pmatrix}
1 & r/p \\ 0 & 1
\end{pmatrix}, \begin{pmatrix}
p^{k\sigma_p(r)} & 0 \\ 0 & p^{-k\sigma_p(r)}
\end{pmatrix} \right\rangle = \langle A, Q_{r/p}, U_p^{k \sigma_p(r)} \rangle
\]
where $\sigma_p(r)$ denotes the multiplicative order of $p$ in $\Z / r\Z$.
\end{corollary}

Notice also that neither the statement of Theorem~\ref{Thm:r-finite-index-number-theory-Artin} nor that of Corollary~\ref{Cor:r-main-generators-artin} contains the restriction $0 \leq \frac{r}{p} \leq 4$. Furthermore, the generating set in Corollary~\ref{Cor:r-main-generators-artin} is close to the generating set for $\Delta_{r/p}$ -- it simply has one more generator $U_p^{\sigma_p(r)}$. Conjecture~\ref{Conj:main-conjecture} amounts to saying that this generator is redundant when $0 < \frac{r}{p} < 4$. When $\frac{r}{p} \geq 4$, the group $\Delta_{r/p}$ is free, so in this case $U_p^{k\sigma_p(r)}$ is never redundant.

\begin{remark}
Artin's conjecture on primitive roots is not known to hold unconditionally for any fixed prime $p$ (but it holds for infinitely many primes \cite{Gupta1984}). Nevertheless, much is known. Heath-Brown \cite{HeathBrown1986} has proved that there are at most two primes for which the conjecture fails. It would not be unreasonable to assume that this also holds true for the residue-wise Artin's conjecture (our Conjecture~\ref{Conj:strong-artin}). If this is true, then we may interpret Corollary~\ref{Cor:r-main-generators-artin} as saying: \textit{for all but at most two primes} $p$, the three matrices $A, Q_{r/p}, U_p^{k \sigma_p(r)}$ generate $\bGamma_1^{(p)}(r)$. Hooley \cite{Hooley1967} has shown that Artin's conjecture follows from the Generalized Riemann Hypothesis. It seems plausible that the same is also true for Conjecture~\ref{Conj:strong-artin}. 
\end{remark}

\clearpage

\section{Verifying Conjecture~\ref{Conj:main-conjecture}}\label{Sec:final-section}

For a prime $p$ and some $r \in \N$ with $0 < \frac{r}{p} < 4$, and $\gcd(r,p)=1$, recall (\S\ref{Sec:Delta-intro}) the definition of $\Delta_{r/p}$ as the subgroup of $\SLP{p}$ generated by the two matrices
\[
A = \begin{pmatrix} 1 & 0 \\ 1 & 1\end{pmatrix}
\text{ and }
Q_{r/p} = \begin{pmatrix} 1 & r/p \\ 0 & 1
\end{pmatrix}
\]
It is clear from the definition of $\bGamma_1^{(p)}(r)$ in \S\ref{Sec:SL2} that $\Delta_{r/p} \leq \bGamma_1^{(p)}(r)$. In this section, we will prove the converse inclusion for some values of $r$ and $p$. By Corollary~\ref{Cor:r-main-generators-artin}, to prove this (conditional on Conjecture~\ref{Conj:strong-artin}) it suffices to verify the following:
\begin{equation}\label{Eq:Up-power-lives-in-delta}
U_p^{k\sigma_p(r)} \in \Delta_{r/p}
\end{equation} 
for some non-zero $k$, where $\sigma_p(r)$ is the multiplicative order of $p$ mod $r$. Notice that, a \textit{fortiori}, if one can verify \eqref{Eq:Up-power-lives-in-delta}, then one has found a non-trivial product of the generators $A$ and $Q_{r/p}$ equalling an upper triangular matrix in $\Delta_{r/p}$. This already implies (as is well-known, see \S\ref{Sec:Delta-intro}) that $\frac{r}{p}$ is a relation number. This is thus a very difficut problem in general. However, by Corollary~\ref{Cor:r-main-generators-artin} (and conditional on Conjecture~\ref{Conj:strong-artin}) verifying \eqref{Eq:Up-power-lives-in-delta} also shows that $\Delta_{r/p}$ has index $J_2(r)$ inside $\SLP{p}$, which is significantly stronger than $\Delta_{r/p}$ being non-free. 

We first verify Conjecture~\ref{Conj:main-conjecture} (unconditionally) for small $r$, using Theorem~\ref{Thm:r-finite-index-small-r}:

\begin{proposition}\label{Cor:small-good-denominators12345}
Let $p$ be a prime, and let $\gcd(r,p)=1$. If $r \in \{ 1, 2, 3, 4 \}$, then
\begin{equation}\label{Eq:delta-rp-equals-gamma1-for-r1234}
\Delta_{r/p} = \bGamma_1^{(p)}(r).
\end{equation}
In particular, in these cases $[\SLP{p} : \Delta_{r/p}] = J_2(r)$.
\end{proposition}
\begin{proof}
The case of $r=1$ was handled already in Proposition~\ref{Prop:r=1}. By Theorem~\ref{Thm:r-finite-index-small-r}, we now have to verify \eqref{Eq:Up-power-lives-in-delta}, i.e.\ that $U_p^{k\sigma_p(r)}$ lies in $\Delta_{r/p}$ for some non-zero $k$. In particular, we guide our search by looking for $U_p^{\sigma_p(r)}$, as this always lies in $\bGamma_1^{(p)}(r)$. Throughout, we will write $Q$ for $Q_{r/p}$. 
\begin{itemize}
\item If $r=2$, then $\sigma_p(r) = 1$, and $U_p \in \Delta_{r/p}$ as 
\[
U_p = \begin{pmatrix} p & 0 \\ 0 & 1/p \end{pmatrix} = Q^s A^{\frac{p-1}{2}} Q^{-1} A^{-\frac{p(p-1)}{2}}.
\]
Note that the condition $\gcd(r,p)=1$ (i.e.\ $p$ is odd) makes all exponents above integral.
\item If $r=3$, then $\sigma_p(r)$ depends on the congruence class of $p \pmod 3$, being either $1$ or $2$: 
\begin{itemize}
\item If $p \equiv 1 \pmod{3}$, then $\sigma_p(3) = 1$, and $U_p \in \Delta_{3/p}$, as
\[
U_p = \begin{pmatrix}
p & 0 \\ 0 & \frac{1}{p}
\end{pmatrix} = Q^{\frac{p(p-1)}{3}} AQ^{-\frac{p-1}{3}} A^{-p}.
\]
\item If $p \equiv 2 \pmod{3}$, then $\sigma_p(3) = 2$, and $U_p^2 \in \Delta_{3/p}$ as
\[
-U_p =  \begin{pmatrix}
-p & 0 \\ 0 & -\frac{1}{p}
\end{pmatrix} = Q^{\frac{p(p+1)}{3}} A^{-1} Q^{\frac{p+1}{3}} A^{-p}.
\]
\end{itemize}
\item If $r=4$, then we again split depending on the congruence class of $p \pmod{4}$, which is either $1$ or $3$. If $p \equiv 1 \pmod{4}$, then $\sigma_p(r) = 1$, so we should have $U_p \in \Delta_{4/p}$. If $p \equiv 3 \pmod{4}$, then $\sigma_p(r) = 2$, so we should have $U_p^2 \in \Delta_{4/p}$. We verify this:
\begin{itemize}
\item If $p \equiv 1 \pmod{4}$, then $U_p \in \Delta_{4/p}$, as $U_p  = Q^p A^{\frac{p-1}{4}} Q^{-1} A^{-\frac{p(p-1)}{4}}$.
\item If $p \equiv 3 \pmod{4}$, then $-U_p \in \Delta_{4/p}$, as $-U_p = Q^{-p} A^{\frac{p+1}{4}} Q^{-1} A^{\frac{p(p+1)}{4}}$. In particular, we have $U_p^2 \in \Delta_{4/p}$. 
\end{itemize}
\end{itemize}
Thus, in all cases we have verified \eqref{Eq:Up-power-lives-in-delta}, so we are done.
\end{proof}

It seems feasible to prove that $\bGamma_1^\Z(5) \leq \Delta_{5/p}$ for all primes $p$ coprime with $5$. We can prove it directly in most cases, using arguments completely analogous to those of Theorem~\ref{Thm:r-finite-index-number-theory-Artin} (and bypassing Artin's conjecture).

\begin{proposition}\label{Prop:r=5-for-p-not-1}
Let $p \neq 5$ be a prime such that $p \not\equiv -1,1 \pmod{11}$. Then
\begin{equation}\label{Eq:delta-rp-equals-gamma1-for-r5}
\Delta_{5/p} = \bGamma_1^{(p)}(r).
\end{equation}
In particular, in these cases $[\SLP{p} : \Delta_{r/p}] = J_2(r)$.
\end{proposition}
\begin{proof}
It is well-known that $\bGamma_1^\Z(5)$ is generated by $A, Q_{5/p}^p$, together with a third matrix, which can be taken as 
\[
M_5 := \begin{pmatrix}
11 & 20 \\ -5 & -9
\end{pmatrix}.
\]
The group of quadratic residues mod $11$ has order $5$, and is hence generated by any non-trivial quadratic residue; since $p \not\equiv -1,1 \pmod{11}$, it follows that $p^2 \not\equiv 1 \pmod{11}$, so $p^2$ generates the quadratic residues mod $11$. Since $\sigma_p(5) \leq 4$ for all $p$, it follows that $p^{2 \sigma_p(5)k}$ generates the quadratic residues mod $11$ for any $k \not\equiv 0 \pmod{5}$. Thus, arguing as in the proof of Theorem~\ref{Thm:r-finite-index-number-theory-Artin}, we have that $M_5 \in \Delta_{5/p}$, and so also $\bGamma_1^\Z(5) \leq \Delta_{5/p}$.

Hence, by Corollary~\ref{Thm:If-contains-GZ-and-power-then-contains-Gp}, it remains to verify that some power of $U_p^{\sigma_p(r)}$ lives in $\Delta_{r/p}$. We do this now. We must split based on the congruence class $\pmod{5}$. First, the case of $p = 2$ is easily dealt with, since $\sigma_2(5) = 4$ and we have 
\[
-U_2^2 = A^2 Q A^{-1} Q A^{-1} Q^{-2}
\]
so $U_2^4 \in \Delta_{5/2}$. Now let $p$ be any odd prime. Then $p$ is one of $\{1, 3, 7, 9\}$ modulo $10$. 
\begin{itemize}
\item If $p \equiv 1 \pmod{10}$, then we have
\[
U_p^2 = A^p Q^{\frac{p-1}{5}} A^{-p-1} Q^{-\frac{p-1}{5}} A Q^{-{\frac{(p+1)p(p-1)^2}{5}}}.
\]
\item If $p \equiv 3 \pmod{10}$, then the matrix
\[
A^p Q^{\frac{2(p-3)}{5} + 1} A^{\frac{p-3}{2} + 1} Q^{-\frac{p-3}{5} - 1} A
\]
is upper triangular with $-p^2$ and $-p^{-2}$ on the diagonal and with upper right entry a multiple of $5p^{-2}$. In particular, right multiplying it by a suitable (rather large) multiple of $Q$ will give $-U_p^2$, and thereby also $U_p^4$.
\item If $p \equiv 7 \pmod{10}$, then the matrix
\[
A^p Q^{-\frac{2(p-7)}{5} - 3} A^{-\frac{p-7}{2} - 3} Q^{-\frac{p-7}{5} - 1} A
\]
is also upper triangular with $-p^2$ and $-p^{-2}$ on the diagonal and with upper right entry a multiple of $5p^{-2}$. Again, we obtain $-U_p^2$, and hence also $U_p^4$.
\item If $p \equiv 9 \pmod{10}$, then the matrix
\[
A^p Q^{-\frac{(p-9)}{5} - 2} A^{p+1} Q^{-\frac{(p-9)}{5} - 2} A
\]
is upper triangular as previously; and we find $-U_p^2$ and $U_p^4$ as before. 
\end{itemize}
This completes the proof of the result.
\end{proof}

A more detailed analysis can doubtlessly be used to eliminate the condition that $p \not\equiv -1, 1 \pmod{11}$ in Proposition~\ref{Prop:r=5-for-p-not-1}. Next, conditional on the strengthened Artin's conjecture, we are able to verify the conditions of Corollary~\ref{Cor:r-main-generators-artin} in some other particular cases, this time being infinite families with arbitrarily large numerators.

\begin{theorem}\label{Thm:easy-numerators-p+1-p-1}
Let $p$ be a prime. Assume the strengthened form of Artin's conjecture on primitive roots (Conjecture~\ref{Conj:strong-artin}) holds for $p$. If one of the following hold:
\begin{enumerate}
\item $r=p-1$;
\item $r=p+1$;
\item $r=\frac{p+1}{2}$ (and $p \neq 2$);
\end{enumerate}
then the group $\Delta_{r/p}$ equals $\bGamma_1^{(p)}(r)$. In particular, $\Delta_{r/p}$ has index $J_2(r)$ in $\SLP{p}$. 
\end{theorem}
\begin{proof}
Throughout this proof, we will let $Q = Q_{r/p}$, where the $r$ is always as given in the individual subcases. 

If $r=p-1$, then since $\sigma_p(p-1) = 1$ it suffices, in view of Corollary~\ref{Cor:r-main-generators-artin}, to verify that $U_p \in \Delta_{(p-1)/p}$. But this is easy to verify: $U_p =  Q^p A Q^{-1} A^{-p}$.

If $r=p+1$, then since $\sigma_p(p+1) = 2$, it suffices to verify that $U_p^2 \in \Delta_{r/p}$. This is also easy to verify, although the length of the shortest product that we have been able to find which verifies this grows cubically in $p$ (the syllable length is fixed):
\[
U_p^2 = \begin{pmatrix}
p^2 & 0 \\ 0 & 1/p^2
\end{pmatrix} = Q^{p^3-p^2-p} A Q^{-1} A^p Q A^{-1} Q A^{-p}.
\]

Finally, if $r = \frac{p+1}{2}$, we first single out the case $p=3$. In this case, $\sigma_p(r) = \sigma_3(2) = 1$, and $U_p = Q^{-3}A^{-1}QA^3$. Suppose now that $p > 3$. Then $\sigma_p(r) = 2$, so we must prove that $U_p^2 \in \Delta_{r/p}$. Since, as one can verify, $-U_p = Q^{-2p}A Q^{-2} A^p$, it follows that $U_p^2 \in \Delta_{r/p}$.
\end{proof}

The rational numbers appearing in Theorem~\ref{Thm:easy-numerators-p+1-p-1} are well-known to be relation numbers, and Theorem~\ref{Thm:easy-numerators-p+1-p-1} also (conditionally) recovers this result.

\subsection{Strong relation numbers}

We will now give another corollary of the methods of this article. As mentioned in \S\ref{Sec:Delta-intro}, to prove that $\Delta_{r/p}$ is not free it suffices to find \textit{some} upper (or lower) triangular matrix as a non-trivial product of $A$ and $Q_{r/p}$. We now introduce a new condition on $r/p$ which strengthens this (although, as we shall conjecture, this is not actually a strengthening). Let $0 < r/p < 4$ with $\gcd(r,p)=1$, i.e.\ let $r/p$ be admissible. Then we say that $r/p$ is a \textit{strong} relation number if there exists some non-trivial product in $A$ and $Q_{r/p}$ which equals an upper (or lower) triangular matrix in which the diagonal elements are not $1$. For example, the number $3/2$ is a strong relation number, as 
\[
A^2 Q^{-1}_{3/2}  A = \begin{pmatrix}
-1/2 & -3/2 \\ 0 & -2
\end{pmatrix}.
\]
To have $r/p$ be a strong relation number is a necessary and sufficient condition for $\Delta_{r/p}$ having finite index in $\SLP{p}$. More precisely, we have:

\begin{theorem}\label{Thm:Delta-conditions-for-finite-index}
 Let $p$ be prime and $\gcd(p,r)=1$. Assume the strengthened form of Artin's conjecture on primitive roots (Conjecture~\ref{Conj:strong-artin}) holds for $p$. Then the following are equivalent: 
\begin{enumerate}
\item $r/p$ is a strong relation number;
\item $U_p^{\sigma_p(r)} \in \Delta_{r/p}$;
\item $\Delta_{r/p} = \bGamma_1^{(p)}(r)$;
\item $[\SLP{p} : \Delta_{r/p}]=J_2(r)$.
\end{enumerate}
\end{theorem}
\begin{proof}
Let us suppose that $r/p$ is a strong relation number, i.e.\ that
\[
M = \begin{pmatrix}
p^k & b \\ 0 & p^{-k}
\end{pmatrix}
\]
lies in $\Delta_{r/p}$ for some non-zero $k$ and $b \in \Z[\frac{1}{p}]$. Then $Q_{r/p} M Q_{r/p}^{-1} = U_p^k$. Since $\Delta_{r/p} \leq \bGamma_1^{(p)}(r)$, it follows that $k$ must be a multiple of the multiplicative order of $p$ $\pmod{r}$. By Corollary~\ref{Cor:r-main-generators-artin}, the result follows. 
\end{proof}

We mention one of the curious consequences of the (still conjectural, in general) equality $\Delta_{r/p} = \bGamma_1^{(p)}(r)$. This is the fact that it implies that $\Delta_{r/p^k} = \Delta_{r/p}$ for all $k \geq 1$, as the generator $Q_{r/p^k}$ is an element of $\bGamma_1^{(p)}(r)$. This can be observed in many cases without any theory; for example, $\Delta_{7/3} = \Delta_{7/3^k}$, as one checks that 
\[
Q_{7/3^k} = \begin{pmatrix}
1 & 7/3^k \\ 0 & 1
\end{pmatrix} = A^{3^k} Q_{7/3}^{-1}A^{3^{k-1}}Q_{7/3}^{-1}A^{3^k}
\]
for all $k \geq 1$. Finding such equalities is very difficult in general, as it involves finding a non-trivial product of $A$ and $Q_{r/p}$ equalling an upper triangular matrix. Nevertheless, it provides targets for \textit{which} upper triangular matrices to generate.

\begin{example}
Assume Conjecture~\ref{Conj:strong-artin} holds. We will prove that $\Delta_{8/3}$ has finite index in $\SLP{3}$ by proving that $8/3$ is a strong relation number. This is direct: 
\[
A^6 Q_{8/3}^{-1}AQ_{8/3}^{-1}A = \begin{pmatrix}
1/9 & 16/9 \\ 0 & 9
\end{pmatrix}.
\]
Hence $8/3$ is a strong relation number, so by Theorem~\ref{Thm:Delta-conditions-for-finite-index} we have that $\Delta_{8/3}$ has index $J_2(8) = 48$ in $\SLP{3}$. This can be verified computationally, see Table~\ref{Tab:cmputational}.
\end{example}

Finally, we list the conjectures about $\Delta_{r/p}$ made (explicitly or implicitly) in this article. Recall that an \textit{admissible} rational $r/p$ is one where $0 \leq r/p \leq 4$. 

\begin{conjecture*}
For all admissible $r, p \in \Z$ with $\gcd(r,p)=1$, the following hold:
\begin{enumerate}[label=\textnormal{(\Roman*)}]
\item $r/p$ is a relation number.
\item $\Delta_{r/p}$ has finite index in $\SLP{p}$ (and is a congruence subgroup).
\item $\Delta_{r/p} = \bGamma_1^{(p)}(r)$, and $[\SLP{p} : \Delta_{r/p}] = J_2(r)$.
\item $U_p^{\sigma_p(r)} \in \Delta_{r/p}$, where $\sigma_p(r)$ is the multiplicative order of $p$ mod $r$.
\item $r/p$ is a strong relation number.
\item $\Delta_{r/p} = \Delta_{r/p^k}$ for all $k \geq 1$. 
\item $\SLP{p} = \Delta_{r/p} \cdot \SL_2(\Z)$. 
\end{enumerate}
\end{conjecture*}

 With the exception of (I), which appears implicitly in Lyndon \& Ullman \cite{Lyndon1969} and explicitly in Kim \& Koberda \cite{Kim2022}, and (II) which appears in \cite{Brody2023}, all the conjectures are original to this article. I.\ Smilga (personal communication) has informed the author that he has a proof of Conjecture~VII, which will appear in forthcoming work.

\begin{remark}
Let us make some remarks on the group generated by the two matrices
\[
X_\mu = \begin{pmatrix}
1 & \mu \\ 0 & 1
\end{pmatrix}, \quad \text{and} \quad Y_\mu = \begin{pmatrix}
1 & 0 \\ \mu & 1
\end{pmatrix}.
\]
Let us write $\Diamond_\mu = \langle X_\mu, Y_{\mu} \rangle$.
Chang, Jennings \& Ree \cite{Chang1958} proved that $\Diamond_\mu \cong \Delta_{\mu^2}$. This isomorphism is not, however, an equality as subsets of $\SL_2(\C)$. If $\mu = r/p$, then similar reasoning as in this article about $\Diamond_{r/p}$ should give similar conjectures as for $\Delta_{r/p}$. For example, it seems as if 
\[
[\SLP{p} : \Diamond_{r/p}] = J_2(r^2)
\]
and that $\Diamond_{r/p}$ consists of the preimage of the matrices in $\SL_2(\Z / r^2 \Z)$ whose diagonal entries are congruent to $1$ mod $r$ (not $r^2$), and whose off-diagonal elements are congruent to $0$ mod $r$. Of course, there are $r^2$ such matrices.

Lyndon \& Ullman \cite{Lyndon1969} in fact seem to have been close to this fact. On \cite[p. 1402]{Lyndon1969}, they consider the group $\Diamond_{3/2}$, and mention that the group is contained in the principal congruence subgroup $\Gamma^{(2)}(3)$, which they denote $\mathbf{U}(3)$, in $\SLP{2}$, but that it does not appear to coincide with it. This is correct; however, on the next line, they write that the group $\Diamond_{3/2}$ contains $\begin{psmallmatrix}
8 & 0 \\ 0 & 1/8
\end{psmallmatrix}$. This is not correct. Indeed, it contradicts the preceding statement: $8 \equiv 2 \pmod{3}$, so $\begin{psmallmatrix}
8 & 0 \\ 0 & 1/8
\end{psmallmatrix}$ cannot possibly be an element of $\Gamma^{(2)}(3)$. On the other hand, $\begin{psmallmatrix}
-8 & 0 \\ 0 & -1/8
\end{psmallmatrix} \in \Gamma^{(2)}(3)$ is correct, and indeed it is not hard to verify that
\[
\begin{pmatrix}
-8 & 0 \\ 0 & -1/8
\end{pmatrix} = X_{3/2}^{12} (Y_{3/2} X_{3/2}^{-1})^2 X_{3/2}^{-1} Y_{3/2}^{-2} \in \Diamond_{3/2},
\]
so it is clear that a minus sign is simply missing from the assertion in Lyndon \& Ullman. It is also not difficult to prove Lyndon \& Ullman's assertion that $\Diamond_{3/2}$ is a proper subgroup of $\Gamma^{(2)}(3)$, and indeed $\Gamma^{(2)}(3) : \Diamond_{3/2}] = 3$. Representatives for these three cosets of $\Diamond_{3/2}$ in $\Gamma^{(2)}(3)$ are easy to construct: they can be taken as $1, \begin{psmallmatrix}
4 & 0 \\ 0 & 1/4
\end{psmallmatrix}$, and $\begin{psmallmatrix}
16 & 0 \\ 0 & 1/16
\end{psmallmatrix}$.
We do not develop this theory further here, but doing this, one is led to entirely analogous conjectures as for $\Delta_{r/p}$. 
\end{remark}

\bibliographystyle{amsalpha}
\bibliography{parabolic-freeness-jan-2024.bib}

\providecommand{\bysame}{\leavevmode\hbox to3em{\hrulefill}\thinspace}
\providecommand{\MR}{\relax\ifhmode\unskip\space\fi MR }
\providecommand{\MRhref}[2]{%
  \href{http://www.ams.org/mathscinet-getitem?mr=#1}{#2}
}
\providecommand{\href}[2]{#2}
\begin{thebibliography}{BFMvL23}

\bibitem[ASD22]{Detinko2022}
A.~Hulpke A.~S.~Detinko, D. L.~Flannery, \emph{Freeness and $s$-arithmeticity
  of rational {M}\"obius groups}, 2022, arXiv:2203.17201.

\bibitem[Bam00]{Bamberg2000}
John Bamberg, \emph{Non-free points for groups generated by a pair of {$2\times
  2$} matrices}, J. London Math. Soc. (2) \textbf{62} (2000), no.~3, 795--801.

\bibitem[BC78]{Brenner1978}
J.~L. Brenner and A.~Charnow, \emph{Free semigroups of {$2\times 2$} matrices},
  Pacific J. Math. \textbf{77} (1978), no.~1, 57--69.

\bibitem[Bey86]{Beyl1986}
F.~Rudolf Beyl, \emph{The {S}chur multiplicator of {${\rm SL}(2,{\bf Z}/m{\bf
  Z})$} and the congruence subgroup property}, Math. Z. \textbf{191} (1986),
  no.~1, 23--42.

\bibitem[BFMvL23]{Brody2023}
Nic Brody, David Fisher, Mahan Mj, and Wouter van Limbeek,
  \emph{Greenberg-{S}halom's commensurator hypothesis and applications}, 2023,
  arXiv:2308.07785.

\bibitem[BM68]{Behr1968}
H.~Behr and J.~Mennicke, \emph{A presentation of the groups {${\rm
  PSL}(2,\,p)$}}, Canadian J. Math. \textbf{20} (1968), 1432--1438.

\bibitem[BMO75]{Brenner1975}
J.~L. Brenner, R.~A. MacLeod, and D.~D. Olesky, \emph{Non-free groups generated
  by two {$2\times 2$} matrices}, Canadian J. Math. \textbf{27} (1975),
  237--245.

\bibitem[Bre55]{Brenner1955}
Jo\"{e}l~Lee Brenner, \emph{Quelques groupes libres de matrices}, C. R. Acad.
  Sci. Paris \textbf{241} (1955), 1689--1691.

\bibitem[Bro82]{Brown1982}
Kenneth~S. Brown, \emph{Cohomology of groups}, Graduate Texts in Mathematics,
  vol.~87, Springer-Verlag, New York-Berlin, 1982.

\bibitem[Cha75]{Charnow1974}
A.~Charnow, \emph{A note on torsion free groups generated by pairs of
  matrices}, Canad. Math. Bull. \textbf{17} (1974/75), no.~5, 747--748.

\bibitem[CJR58]{Chang1958}
Bomshik Chang, S.~A. Jennings, and Rimhak Ree, \emph{On certain pairs of
  matrices which generate free groups}, Canadian J. Math. \textbf{10} (1958),
  279--284.

\bibitem[FR40]{FuchsRabinowitsch1940}
D.~J. Fuchs-Rabinowitsch, \emph{On a certain representation of a free group},
  Leningrad State Univ. Annals [Uchenye Zapiski] Math. Ser. \textbf{10} (1940),
  154--157.

\bibitem[Gil08]{Gilman2008}
Jane Gilman, \emph{The structure of two-parabolic space: parabolic dust and
  iteration}, Geom. Dedicata \textbf{131} (2008), 27--48. \MR{2369190}

\bibitem[GM84]{Gupta1984}
Rajiv Gupta and M.~Ram Murty, \emph{A remark on {A}rtin's conjecture}, Invent.
  Math. \textbf{78} (1984), no.~1, 127--130.

\bibitem[HB86]{HeathBrown1986}
D.~R. Heath-Brown, \emph{Artin's conjecture for primitive roots}, Quart. J.
  Math. Oxford Ser. (2) \textbf{37} (1986), no.~145, 27--38.

\bibitem[Hoo67]{Hooley1967}
Christopher Hooley, \emph{On {A}rtin's conjecture}, J. Reine Angew. Math.
  \textbf{225} (1967), 209--220.

\bibitem[Ign80a]{Ignatov1980}
Ju.~A. Ignatov, \emph{Groups of linear fractional transformations generated by
  three elements}, Mat. Zametki \textbf{27} (1980), no.~4, 507--513, 668.

\bibitem[Ign80b]{Ignatov1980b}
\bysame, \emph{Roots of unity as nonfree points of the complex plane}, Mat.
  Zametki \textbf{27} (1980), no.~5, 825--827, 831.

\bibitem[Jor70]{Jordan1870}
Camille Jordan, \emph{Trait\'{e} des substitutions et des \'{e}quations
  alg\'{e}briques}, Gauthier-Villars, 1870.

\bibitem[KK22]{Kim2022}
Sang-hyun Kim and Thomas Koberda, \emph{Non-freeness of groups generated by two
  parabolic elements with small rational parameters}, Michigan Math. J.
  \textbf{71} (2022), no.~4, 809--833. \MR{4505367}

\bibitem[KS94]{Keen1994}
Linda Keen and Caroline Series, \emph{The {R}iley slice of {S}chottky space},
  Proc. London Math. Soc. (3) \textbf{69} (1994), no.~1, 72--90. \MR{1272421}

\bibitem[LU69]{Lyndon1969}
R.~C. Lyndon and J.~L. Ullman, \emph{Groups generated by two parabolic linear
  fractional transformations}, Canadian J. Math. \textbf{21} (1969),
  1388--1403.

\bibitem[Mag31]{Magnus1931}
Wilhelm Magnus, \emph{Untersuchungen \"{u}ber einige unendliche
  diskontinuierliche {G}ruppen}, Math. Ann. \textbf{105} (1931), no.~1, 52--74.

\bibitem[Mar47]{Markov1947}
A.~Markov, \emph{On certain insoluble problems concerning matrices}, Doklady
  Akad. Nauk SSSR (N.S.) \textbf{57} (1947), 539--542.

\bibitem[Men67]{Mennicke1967}
J.~Mennicke, \emph{On {I}hara's modular group}, Invent. Math. \textbf{4}
  (1967), 202--228.

\bibitem[MK23]{Kourovka2023}
V.~D. Mazurov and E.~I. Khukhro, \emph{Unsolved {P}roblems in {G}roup {T}heory,
  the {K}ourovka {N}otebook (no. 20)}, 2023.

\bibitem[Mor12]{Moree2012}
Pieter Moree, \emph{Artin's primitive root conjecture---a survey}, Integers
  \textbf{12} (2012), no.~6, 1305--1416. \MR{3011564}

\bibitem[New74]{Canberra1974}
M.~F. Newman (ed.), \emph{Proceedings of the {S}econd {I}nternational
  {C}onference on the {T}heory of {G}roups}, Lecture Notes in Mathematics, Vol.
  372, Springer-Verlag, Berlin-New York, 1974, Held at the Australian National
  University, Canberra, August 13--24, 1973, With an introduction by B. H.
  Neumann.

\bibitem[Ree61]{Ree1961}
Rimhak Ree, \emph{On certain pairs of matrices which do not generate a free
  group}, Canad. Math. Bull. \textbf{4} (1961), 49--52.

\bibitem[RM68]{Ree1968}
Rimhak Ree and N.~S. Mendelsohn, \emph{Free subgroups of groups with a single
  defining relation}, Arch. Math. (Basel) \textbf{19} (1968), 577--580 (1969).

\bibitem[San47]{Sanov1947}
I.~N. Sanov, \emph{A property of a representation of a free group}, Doklady
  Akad. Nauk SSSR (N.S.) \textbf{57} (1947), 657--659.

\bibitem[Ser70]{Serre1970}
Jean-Pierre Serre, \emph{Le probl\`eme des groupes de congruence pour {SL}2},
  Ann. of Math. (2) \textbf{92} (1970), 489--527. \MR{272790}

\bibitem[Ser80]{Serre1980}
\bysame, \emph{Trees}, Springer-Verlag, Berlin-New York, 1980, Translated from
  the French by John Stillwell.

\bibitem[S{\l}a16]{Slanina2016}
Piotr S{\l}anina, \emph{Generalizations of {F}ibonacci polynomials and free
  linear groups}, Linear Multilinear Algebra \textbf{64} (2016), no.~2,
  187--195.

\bibitem[Smi21]{Smilga2021}
Ilia Smilga, \emph{New sequences of non-free rational points}, C. R. Math.
  Acad. Sci. Paris \textbf{359} (2021), 983--989.

\bibitem[Spe60]{Specht1959}
Wilhelm Specht, \emph{Freie {U}ntergruppen der bin\"{a}ren unimodularen
  {G}ruppe}, Math. Z. \textbf{72} (1959/60), 319--331.

\bibitem[Zub98]{Zubkov1998}
A.~N. Zubkov, \emph{On a matrix representation of a free group}, Mat. Zametki
  \textbf{64} (1998), no.~6, 863--870.

\end{thebibliography}

 \end{document}